\documentclass[11pt]{article}%
\usepackage{amssymb, amsmath, latexsym, mathrsfs, verbatim, calc}
\usepackage[latin1]{inputenc}
\usepackage{amsmath}
\usepackage{amsfonts}
\usepackage{amssymb}
\usepackage{graphicx}
\usepackage{color}%
\providecommand{\U}[1]{\protect\rule{.1in}{.1in}}

\newtheorem{theorem}{Theorem}

\newtheorem{definition}[theorem]{Definition}
\newtheorem{example}[theorem]{Example}

\newtheorem{lemma}[theorem]{Lemma}

\newtheorem{remark}[theorem]{Remark}

\newenvironment{proof}[1][Proof]{\noindent\textbf{#1.} }{\ \rule{0.5em}{0.5em}}

\def\U{\mathbb{U}}
\begin{document}

\title{On a nonlinear neutral stochastic functional integro-differential
	equation driven by fractional Brownian motion}

\author{B. BOUFOUSSI \thanks{Department of Mathematics, Faculty of Sciences Semlalia, Cadi Ayyad University, 2390 Marrakesh, Morocco}
\and S. Hajji \thanks{Department of Mathematics, Regional Center for the Professions of Education and Training, Marrakesh, Morocco}
\and S. MOUCHTABIH \footnotemark[1]\quad \thanks{Corresponding author. Email: soufiane.mouchtabih@gmail.com} }

\date{}
\maketitle

\vskip -0.5cm


\vskip -0.4cm

\noindent\textbf{Abstract}: In this paper, we study the existence and uniqueness of mild solution for a stochastic neutral partial functional integro-differential equation with delay in a Hilbert space driven by a fractional Brownian motion and with non-deterministic diffusion coefficient. We suppose that the linear part has a resolvent operator in the sense given in \cite{Grimmer}. We also establish a sufficient condition for the existence of the density of a function of the solution. An example is provided to illustrate the results of this work.\bigskip

 \bigskip
 

\section{Introduction}

The theory of stochastic differential equations driven by a fractional Brownian motion in infinite dimension has been widely examined by many researchers
due to various mathematical models in different fields, such as physics, population dynamics, ecology, biological systems, biotechnology, optimal control, theory of elasticity,
electrical networks, and several other areas of engineering and science. There is now a rich litterature on the topic of stochastic equations driven by a fractional Brownian motion,
for example, \cite{duncan2}, proved the existence and uniqueness of a mild solution for a class of stochastic differential equations in a Hilbert space with a cylindrical fractional Brownian motion with Hurst parameter $H\in (\frac{1}{2},1)$, \cite{MaslNual} established the existence and uniqueness of a mild solution for non linear stochastic evolution equations in Hilbert space. Moreover, \cite{Caraballo2011} discussed mild solutions for stochastic delay evolution equations driven by a fractional Brownian motion.  \\The study of neutral stochastic functional differential equations (NFSDEs) driven by a fractional Brownian motion (fBm) has became an active area of investigation.
\cite{Boufoussi} analysed the existence and uniqueness of a mild solution for a neutral stochastic differential equation with finite delay driven by fractional Brownian motion in a Hilbert space, and established some sufficient conditions ensuring the exponential decay to zero in mean square for the solution. In \cite{Caraballo} the authors studied the existence
and uniqueness of mild solutions to neutral stochastic delay functional integro-differential equations perturbed by fractional Brownian motion. Since then, many other works have followed dealing with the same subject, but all these works consider NFSDEs with deterministic diffusion coefficients. Recently, \cite{Boufoussi2}, inspired by the work of \cite{ferrante}, proved the existence of mild solution to delay differential equations driven by a fractional Brownian motion in Hilbert space and with nonlinear stochastic diffusion terms. Following this line, our main objective in this paper, is to generalize the result in \cite{Caraballo} to a class of neutral stochastic functional integro-differential equations with nondeterministic diffusion terms. More precisely, we consider the following stochastic equation:
\begin{eqnarray}\label{eq1}
	&& d[x(t)+g(x(t-r))]=A[x(t)+g(x(t-r))]dt\nonumber\\
	&&\qquad\qquad +\left[\int_0^tB(t -s)\left[x(s)+g(x(s-r))\right]ds+ f(x(t-r))\right]dt\nonumber\\
	&&\qquad \qquad+\sigma(x(t-r))dB^H(t),\qquad\qquad\qquad\qquad\qquad\qquad\;0\leq t \leq T,\nonumber\\
	x(t)&=&\varphi (t), \; -\tau \leq t \leq 0\,,
\end{eqnarray}
where $ A $ is a closed linear operator on a  Hilbert space $V$ with domain $D(A)\subset V $.
For all $t\geq 0,\, B(t)$ is a closed linear operator with domain $ D(B)\supset D(A)$ independent
of $t$. $B^H$ is a fractional Brownian motion
with Hurst parameter $H>1/2$ defined in a complete probability
space $(\Omega,\mathcal{F},\mathbb{P})$,\, $f,\sigma,\, g:V \rightarrow V$ are
appropriate functions, while $\varphi:[-r,0]\rightarrow V$ is a continuous function.  The importance of considering this equation lies on the difficulties caused by the presence of
the stochastic integral's term, which is defined here in the Skorohod sense and then have to be carefully managed. We prove the existence and uniqueness of a mild solution of the equation (\ref{eq1}). The proof is based on an iterative procedure to estimate the Malliavin derivatives of the solution. The technical nature of this method makes it difficult to answer certain classical questions such as studying the stability of the solution.
\\
The theory of integro-differential equations with resolvent operators is an important branch of differential equations, which has an extensive physical background and provides useful mathematical models for many fields of applications. This is why it has received much attention in recent years. The resolvent operator is similar to the semigroup operator for abstract differential equations in Banach spaces. However, the resolvent operator does not satisfy semigroup properties.

This paper is organized as follows, In Section 2 we  introduce some notations, concepts, and basic results about fractional Brownian motion and Stochastic integral on Hilbert spaces. The existence and uniqueness of mild solutions are discussed in Section 3. In Section 4, we investigate the absolute continuity for the law of
$F(x(t))$, where $x(t)$ is the mild solution obtained in section $3$ and $F$ a function satisfying appropriate conditions.
Finally, in section $5$ we will exhibit an example to illustrate our previous abstract results.

\section{Preliminaries}

In this section, stochastic integrals with respect to a fractional Brownian motion in a separable Hilbert space  is introduced, and some basic properties of this integral are noted.\\
Let $(\Omega,\mathcal{F}, \mathbb{P})$ be a complete probability space.
Consider a time interval $[0,T]$ with arbitrary fixed horizon $T$ and let
$\{B^H(t) , t \in [0, T ]\}$ the one-dimensional fractional Brownian motion
with Hurst parameter $H\in(1/2,1)$. This means by definition that $B^H$ is a centred Gaussian process with covariance function:
$$ R_H(s, t) =\frac{1}{2}(t^{2H} + s^{2H}-|t-s|^{2H}).$$
Moreover $B^H$ has the following Wiener integral representation:
$$B^H(t) =\int_0^tK_H(t,s)\, dB(s)$$
where $B = \{B(t) :\; t\in [0,T]\}$ is a Wiener process, and $K_H(t; s)$ is
the kernel given by
$$K_H(t, s )=c_Hs^{\frac{1}{2}-H}\int_s^t (u-s)^{H-\frac{3}{2}}u^{H-\frac{1}{2}}du$$
for $t>s$, where $c_H=\sqrt{\frac{H(2H-1)}{\beta (2-2H,H-\frac{1}{2})}}$ and $\beta(,)$ denotes the Beta function. We put $K_H(t, s ) =0$ if $t\leq s$.\\
We will denote by $\mathcal{H}$ the reproducing kernel Hilbert space of the fBm. In fact
$\mathcal{H}$ is the closure of set of indicator functions $\{1_{[0;t]},  t\in[0,T]\}$ with respect to the scalar product
$$\langle 1_{[0,t]},1_{[0,s]}\rangle _{\mathcal{H}}=R_H(t , s).$$
The mapping $1_{[0,t]}\rightarrow B^H(t)$ can be extended to an isometry between $\mathcal{H}$
and the first  Wiener chaos and we will denote by $B^H(\varphi)$ the image of $\varphi$ by the previous isometry.

We recall that for $\psi,\varphi \in \mathcal{H}$ their scalar product in
$\mathcal{H}$ is given by
$$\langle \psi,\varphi\rangle _{\mathcal{H}}=H(2H-1)\int_0^T\int_0^T\psi(s)
\varphi(t)|t-s|^{2H-2}dsdt\,.$$
Let us consider the operator $K_H^*$ from $\mathcal{H}$ to $L^2([0,T])$ defined by
$$(K_H^*\varphi)(s)=\int_s^T\varphi(r)\frac{\partial K}{\partial r}(r,s)dr\,.$$
We refer to \cite{nualart} for the proof of the fact that $K_H^*$ is an isometry between $\mathcal{H}$ and $L^2([0,T])$. Moreover for any $\varphi \in \mathcal{H}$, we have
$$ B^H(\varphi)=\int_0^T(K_H^*\varphi)(t)\, dB(t)\,.$$
Let $\mathcal{S}$ denote the class of smooth random variables such that a random variable $F \in \mathcal{S}$ has the form
\begin{equation}\label{l1}
	F = f(B^H(\phi_1), . . . , B^H(\phi_n))
\end{equation}
where $f$ belongs to $\mathcal{C}_b^{\infty}(\mathbb{R}^n),\; \phi_1, . . . , \phi_n$ are in $\mathcal{H}$ and $n \geq 1$.
The derivative of a smooth random variable $F$ of the form $(\ref{l1})$ is the $\mathcal{H}$-valued random variable given by
$$DF =\sum_{j=1}^n \frac{\partial f}{\partial x_j}(B^H(\phi_1), . . . , B^H(\phi_n))\phi_j.$$
We can define the iteration of the operator $D$ in such a way that for a smooth random variable $F$, the iterated derivative $D^kF$ is a random variable with values in $\mathcal{H}^{\otimes k}$.

Let $V$ be a real separable Hilbert space. We denote by $\mathcal{H}_V$ the completion of the pre-Hilbert space of $V$-valued bounded Borel measurable functions $F:[0,T]\rightarrow V$ with the norm induced from the inner product
$$<F, G>_{\mathcal{H}_V}:=\int_0^T\int_0^T\langle F(s),G(t)\rangle_V \phi_H(t-s)ds\,dt$$
where $\phi_H(s)=H(2H-1)|s|^{2H-2}$. It is easily verified that $\mathcal{\widetilde{H}}_V  \subseteq \mathcal{H}_V,$ where $\mathcal{\widetilde{H}}_V$ is the Banach space of Borel measurable functions with the norm $\mid.\mid_{\mathcal{\widetilde{H}}_V}$ given by
$$\mid \varphi\mid^2_{\mathcal{\widetilde{H}}_V}:= \int_0^T\int_0^T \parallel
\varphi(u)\parallel  \parallel \varphi(v) \parallel\phi_H(u-v)du dv\,.$$
It can be seen that $\mathbb{L}^{1/H}([0,T],V)\subseteq \mathcal{\widetilde{H}}_V$ and, in particular, $ \mathbb{L}^{2}([0,T],V)
\subseteq\mathcal{\widetilde{H}}_V$. For more details we can refer to \cite{duncan1}.

Consider the family $\mathcal{S}_V$ of $V$-valued smooth random variables of
the form
$$F=\sum_{j=1}^n F_j v_j,\;\;\; v_j\in V,\;\; F_j\in \mathcal{S}\,. $$
Define $D^kF=\sum_{j=1}^n D^kF_j\otimes v_j,$ then $D^k$ is a closable operator from $\mathcal{S}_V \subset \mathbb{L}^p(\Omega,V)$ into $\mathbb{L}^p(\Omega,\mathcal{H}_V^{\otimes k})$ for any $p\geq 1$.\\
For any integer $k\geq 1$ and any real number $p\geq 1,$ we can define the semi-norm on $\mathcal{S}_V $
$$\|F\|_{k,p}^p=E\|F\|_V^p+E(\sum_{i=1}^k \|D^iF\|^p_{\mathcal{H}_V^{\otimes i}})\,.$$
We define the space $\mathbb{D}^{k,p}(V)$ as the completion of $\mathcal{S}_V$ with respect to the norm $\|F\|_{k,p}$, and we will write  $\mathbb{D}^{k,p}(\mathbb{R})=\mathbb{D}^{k,p}.$
\begin{definition}
	The divergence operator $\delta$ is the adjoint of the derivative operator, defined by means of the duality relationship
	$$E\langle F, \delta (u) \rangle_V=E\langle DF, u\rangle_{\mathcal{H}_V}, \forall F\in \mathcal{S}_V$$
	where $u$ is a random variable in $\mathbb{L}^2(\Omega,\mathcal{H}_V)$.
\end{definition}
\begin{definition}\label{def integr}
	We say that $u$ belongs to the domain of the operator $\delta$, denoted by $Dom(\delta)$, if $F\mapsto \langle u,DF\rangle_{\mathcal{H}_V}$ is continuous on $\mathcal{S}_V$ with respect to the $\mathbb{L}^2(\Omega)$ norm.\\
	For $u \in Dom(\delta) $, we denote $\delta (u)$ by $\int_0^T u(s) \, dB^H(s)$.
\end{definition}
Let us now consider the space $\widetilde{\mathcal{H}}_V\otimes
\widetilde{\mathcal{H}}_V$ of measurable functions $\varphi:[0,T]^2\rightarrow V$
such that
$$
\|\varphi\|^2_{\widetilde{\mathcal{H}}_V^{\otimes 2}}=\int_{[0,T]^4}
\|\varphi(s,v)\|\|\varphi(s',v')\|\phi_H(s-s')\phi_H(v-v')ds\,
dv\,ds'\,dv'<\infty\,.
$$
We denote by $\mathbb{D}^{1,2}(\widetilde{\mathcal{H}}_V)$ the space of processes $ u $ satisfying:
$$E\|u\|^2_{\widetilde{\mathcal{H}}_V}+E\|Du\|^2_{\widetilde
	{\mathcal{H}}_V^{\otimes 2}} < \infty$$
It has been shown in \cite{duncan1} that $\mathbb{D}^{1,2}(\widetilde{\mathcal{H}}_V)$ is included in $Dom (\delta)$ and for a process $u$ in $\mathbb{D}^{1,2}(\widetilde{\mathcal{H}}_V)$
we have:
$$ E\parallel \delta(u)\parallel^2=E\|u\|^2_{\mathcal{H}_V}+E\int_{[0,T]^4}\langle D_pu(q), D_v u(s) \rangle
\phi_H(p-s)\phi_H(v-q)dp\,dq\,dv\,ds.$$

For fixed $m\geq 1$, we will say that a $V -$valued stochastic process  $Z=\{Z(t),
t\in [0,T]\}$ satisfies condition $\mathcal{A}_{m}$ if $Z(t)\in \mathbb{D}^{m,2}(V)$ for any $t \in [0,T]$, and for any
$k\leq m $, we have
$$\displaystyle \sup_t \mathbb{E}\|Z(t)\|^2_V \leq c_{1} \;\;
\mbox{ and}\;\;\displaystyle \sup_t\sup_{u,|u|=k}\mathbb{E}\|D_u^kZ(t)\|^2_V\leq
c_{2,k}$$
for some positive constants $c_{1},c_{2,k}$, and where $\displaystyle\sup_{u,|u|=k} $ is the supremum taken on all the vectors $u$ of length $k$.
Notice that if $Z$ satisfies condition $\mathcal{A}_{m}$, then $Z$ belongs to $\mathbb{D}^{1,2}(\mathcal{\widetilde{H}}_V) $.

Now we turn to state some notations and basic facts about  the theory of resolvent operators needed in the sequel. For additional details on resolvent operators, we refer  to  \cite{Grimmer} and \cite{pruss}.\\
Let $A:D(A)\subset V \rightarrow V$  be a closed linear operator and  for all $t\geq 0,\, B(t)$  a closed linear operator with domain $ D(B(t))\supset D(A)$.
Let us denote by $X$ the Banach space $D(A)$, the domain of operator $A$, equipped with the graph norm
$$\|y\|_X :=\|Ay\|_V+\|y\|_V \;\;\mbox{for}\;\; y\in X.$$
We will note by $C([0,+\infty),X)$, the space of all continuous functions from $[0,+\infty)$ into $X$, and $\mathcal{B}(X,V)$ the set of all bounded linear operators form $X$ into $V$. Consider the following Cauchy problem
\begin{eqnarray}\label{cauchy}
	v'(t) &=& Av(t)+\int_0^tB(t -s)v(s)ds \,,\,\;\; \mbox{for}\;\; t\geq 0,\nonumber\\
	v(0) &=& v_0 \in V.
\end{eqnarray}
We recall the following definition (\cite{Grimmer})  
\begin{definition} A resolvent operator of the
	Equation $(\ref{cauchy})$ is a bounded linear operator valued function $R(t)\in \mathcal{L}(V)$ for $t\geq 0$, satisfying the following properties:
	\begin{itemize}
		\item  [(i)] $ R(0) = I$ and $\|R(t)\|\leq Ne^{\beta t}$  for some constants $N$ and $\beta$.
		\item  [(ii)] For each $x\in V$, $R(t)x$ is strongly continuous for $t\geq 0$.
		\item   [(iii)] $R(t) \in \mathcal{L}(X) $ for $t\ge 0$. For $x \in X$, $R(.)x\in \mathcal{C}^1([0,+\infty);V)\cap \mathcal{C}([0,+\infty);X)$ and
		$$R'(t)x = AR(t)x +\int_0^tB(t -s)R(s)xds= R(t)Ax+\int_0^tR(t -s)B(s)xds, \;\;\mbox{for}\;\; t\geq 0.$$
	\end{itemize}
\end{definition}
The resolvent operator satisfies a number properties reminiscent semi-group, it plays an important role to study the existence of solutions and to establish a variation of constants formula for nonlinear systems. To assure the existence of the resolvent operator, we need the following hypotheses:
\begin{itemize}
	\item [$(\mathcal{H}.1)$] $A$ is the infinitesimal generator of a $C_0$-semigroup $(T(t))_{t\geq0}$ on $V$.
	\item [$(\mathcal{H}.2)$] For all $t\geq 0$, $B(t)$ is a closed linear operator from $X$
	into $V$. Furthermore for any $y\in X$ the map $t\to B(t)y$ is bounded, differentiable and the derivative $t\to B'(t)y$ is bounded and uniformly continuous on $\mathbb{R}^+$.  
\end{itemize}
The following theorem gives the existence conditions of a resolvent operator for the equation (\ref{cauchy}).
\begin{theorem}(\cite{Grimmer}) \label{res}
	Assume that hypotheses $(\mathcal{H}.1)$ and $(\mathcal{H}.2)$ hold, then the Cauchy problem  $(\ref{cauchy})$  admits a unique resolvent operator $(R(t))_{t\geq 0}$.
\end{theorem}
In what follows, we recall some existence results for the following integro-differential equation
\begin{eqnarray}\label{Cauchy-integro}
	v'(t)&=& Av(t)+ \int_0^t B(t-s)v(s)ds+q(t),\ for t\ge 0\\
	v(0)&=&v_0\in V\nonumber
\end{eqnarray}
where $q:[0,+\infty[\to V$ is a continuous function.\\
\begin{definition}
	A continuous function $v:[0,+\infty)\to V$ is said to be a strict solution of equation $(\ref{Cauchy-integro})$ if
	\begin{itemize}
		\item[$(i)$] $v\in C^1([0,+\infty),V)\cap C([0,+\infty),X)$,
		\item[$(ii)$] $v$ satisfies equation $(\ref{Cauchy-integro})$ for $t\ge 0$.
	\end{itemize}
\end{definition}
\begin{definition}
	A function $v:[0,+\infty)\to V$ is called a mild solution of $(\ref{Cauchy-integro})$ if it satisfies the following variation of constants formula: for any $v(0)\in V$,
	\begin{equation}\label{variation-cons}
		v(t)=R(t)v(0)+\int_0^t R(t-s)q(s)ds\,,\,\,\,\, t\geq 0\,,
	\end{equation}
	where $R(t)$ is the resolvent operator of the Equation (\ref{cauchy}).
\end{definition}
\begin{remark}
	It has been shown in \cite{Grimmer} that under $(\mathcal{H}.1)$ and $ (\mathcal{H}.2)$ a strict solution of Equation $(\ref{Cauchy-integro})$ is a mild solution. Reciprocally, if in addition the function $q$ is sufficiently regular a mild solution of Equation $(\ref{Cauchy-integro})$ with $v_{0}\in D(A)$ is a strict solution.\\
	Clearly in our situation, due to the presence of the stochastic integral in the Equation (\ref{eq1}), we will not be concerned by strict solutions.
\end{remark}
\section{Existence and uniqueness of a mild solution}
In this section we study the existence and uniqueness of a mild solution for Equation (\ref{eq1}). In the sequel, we assume that the following conditions hold.
\begin{itemize}
	\item [$(\mathcal{H}.3)$]
	$f,\, g,\, \sigma:V \rightarrow V $  are bounded functions with bounded Fr\'echet
	derivatives up to some order $m\geq1.$
\end{itemize}
Moreover, we assume that $\varphi \in \mathcal{C}([-\tau,0],\mathbb{L}^2(\Omega,V))$.

Similar to the deterministic situation we give the following definition of mild solutions for Equation (\ref{eq1}).
\begin{definition}
	An $V$-valued  process $\{x(t),\;t\in[-\tau,T]\}$, is called  a mild solution of equation (\ref{eq1}) if
	\begin{itemize}
		\item[$i)$] $x(.)\in \mathcal{C}([-r,T],\mathbb{L}^2(\Omega,V))$,
		\item[$ii)$] $x(t)=\varphi(t), \, -r \leq t \leq 0$.
		\item[$iii)$]For arbitrary $t \in [0,T]$, we have
		\begin{eqnarray*}
			x(t)&=&R(t)[\varphi(0)+g(\varphi(-r))]-g(x(t-r))\\
			&+&\int_0^t R(t-s)f(x(s-r))ds+\int_0^t R(t-s)\sigma(x(s-r))dB^H(s)\;\;
			\mathbb{P}-a.s.\phantom{\int_0^2+2}\,,\\
		\end{eqnarray*}
	\end{itemize}
	where $R(.)$ is the resolvent operator of the Cauchy problem (\ref{cauchy}).
\end{definition}
Our main result is the following:
\begin{theorem}\label{thm1}
	Suppose that $(\mathcal{H}.1)$, $(\mathcal{H}.2)$ and $(\mathcal{H}.3)$ hold. Then the equation $(\ref{eq1})$ admits a unique solution on $[-r,T]$ for every $T\le mr$.
\end{theorem}
For the proof we need the following lemma which can be proved by the same arguments as those used in \cite{Boufoussi2}. 
\begin{lemma}\label{lem1}
	Let $y=\{y(t),t\in [0,T]\}$ be a stochastic process.
	\begin{itemize}
		\item [$(i)$] If $y$ satisfies condition $\mathcal{A}_m$ and if $b:V\to V$ is a bounded function with bounded derivatives up to order $m$.
		Then, the stochastic process\\ $\{Z(t)=b(y(t)), t\in [0,T]\}$ satisfies condition $\mathcal{A}_m$.
		\item  [$(ii)$] If $y$ satisfies condition $\mathcal{A}_m$, then, the stochastic process\\ $\{Z(t)=\int_0^t R(t-s)y(s)\,ds,\, t\in [0,T]\}$ satisfies condition $\mathcal{A}_m$.
		\item [$(iii)$] If $y$ satisfies condition $\mathcal{A}_{m+1}$, then,the stochastic Skorohod integral\\ $\{Z(t)=\int_0^t R(t-s)y(s)\, dB^H(s), t\in [0,T]\}$
		is well defined and the stochastic process $Z=\{Z(t),\,t\in [0,T]\}$ satisfies condition $\mathcal{A}_m$.
	\end{itemize}
\end{lemma}

\begin{proof} of Theorem \ref{thm1}.
	To prove that the equation $(\ref{eq1})$ admits a unique solution on $[0,T]$, with $T\leq mr $, we construct the solution step by step. Let us consider the induction hypothesis $(H_n)$ for $1\leq n\le m$:\\
	$ (H_n) $: The equation
	\begin{eqnarray}\label{eq related to Hn}
		x_{}(t)&=& R(t)[\varphi(0)+g(\varphi(-r))]-g(x_{}(t-r))+\int_0^tR(t-s)f(x_{}(s-r))ds\nonumber\\
		&+&\int_0^tR(t-s)\sigma(x_{}(s-r))\, dB^H(s)
		\,,\,\,\,\,\quad t\in [0,nr]\ \\
		x(t)&=&\varphi(t),\,\quad  t\in [-r,0]\nonumber
	\end{eqnarray}
	has a unique solution $x_{n}(t)$ which satisfies condition
	$\mathcal{A}_{m-n}$.\\
	Let us check $(H_1)$. For $t\in [0,r]$, equation $(\ref{eq related to Hn})$ can be written in the following form:
	\begin{eqnarray*}
		x_{1}(t)&=&\varphi(t),\,\quad  t\in [-r,0]\\
		x_{1}(t)&=& R(t)[\varphi(0)+g(\varphi(-r))]-g(\varphi(t-r))\\
		&+&\int_0^tR(t-s)f(\varphi(s-r))ds+\int_0^tR(t-s)\sigma(\varphi(s-r))\, dB^H(s)
		\,,\,\,\,\,\quad t\in [0,r]\ \\
	\end{eqnarray*}
	Since $\varphi$ is a deterministic continuous function, it follows that $x_1 (t)\in \mathbb{D}^{k,2}(V)$ for all $k\ge 1$. Therefore,
	\begin{eqnarray*}
		D_u x_1(t)=R(t-u)\sigma(\varphi(u-r))1_{u<t<r}
	\end{eqnarray*}
	and then $D^k x_1 (t)=0$ when $k\ge 2$. Thanks to the boundedness of the coefficients $f$, $g$ and $\sigma$ we can easily check that
	\begin{eqnarray*}
		\sup_t\mathbb{E}\|x_1 (t)\|^2\le c_1,\quad \sup_t\sup_{u,|u|=k}\mathbb{E}
		\|D^k_u x_1 (t)\|^2\le c_{2,k}
	\end{eqnarray*}
	Then $x_1 $ satisfies condition $\mathcal{A}_k$ for any $k\ge 1$, hence $x_1 $ satisfies $\mathcal{A}_{m-1}$.\\
	Assume now that $(H_n)$ is true for $n<m$, and check $(H_{n+1})$.
	Consider the stochastic process $\{x_{n+1}(t),t\in [-r,(n+1)r]\}$ defined as:
	\begin{eqnarray}\label{def of x_{n+1}}
		x_{n+1}(t)&=& R(t)[\varphi(0)+g(\varphi(-r))]-g(x_n(t-r))\nonumber\\
		&+&\int_0^tR(t-s)f(x_n(s-r))ds+\int_0^tR(t-s)\sigma(x_n(s-r))dB^H_s,\\
		x_{n+1}(t)&=&\varphi(t),\,\quad  t\in [-r,0]\,,\nonumber 
	\end{eqnarray}
	where $x_n$ is the solution obtained in $(H_n)$. The process $x_{n+1}$ is well defined, thanks to fact that $x_n$ satisfies $\mathcal{A}_{m-n}$, assumption $(\mathcal{H}.3)$,
	the boundedness of $R$ and Lemma \ref{lem1}. Moreover $x_{n+1}$ verifies $\mathcal{A}_{m-n-1}$.\\
	Therefore, for $t\le nr$, the uniqueness of the solution on $[0, nr]$, entails: 
	$$x_{n+1}(t)=x_n(t).$$
	Then, equation $(\ref{def of x_{n+1}})$ becomes:
	\begin{eqnarray*}
		x_{n+1}(t)&=& R(t)[\varphi(0)+g(\varphi(-r))]-g(x_{n+1}(t-r))\nonumber\\
		&+&\int_0^tR(t-s)f(x_{n+1}(s-r))ds+\int_0^tR(t-s)\sigma(x_{n+1}(s-r))dB^H_s.
	\end{eqnarray*}
	
	Finally, $x_{n+1}$ is the unique solution of equation $(\ref{eq related to Hn})$ on $[0,(n+1)r]$. The procedure is verified up to $n=m$, and the process $x(t)=x_m(t)$ is the unique solution of Equation $(\ref{eq1})$ on $[-r,T]$ for all $T\le mr$. We can easily check the continuity of the solution, that is $x(.)\in\mathcal{C}([-r,T],\mathbb{L}^2(\Omega,V))$. Which ends the proof of the theorem.
\end{proof}

\section{Regularity of the law of $F(x(t))$}
In this section, we find a sufficient conditions under which the law of $F(x(t))$ is absolutely continuous with respect to the Lebesgue measure, where $x(t)$ is the solution of Equation $(\ref{eq1})$ and $F$ is a real Lipschitzian function. More precisely, we have the following result:
\begin{theorem}
	Assume that the hypothesis of Theorem \ref{thm1} hold and let $\{x(t),t \in [0,T]\}$ be the solution of the equation (\ref{eq1}) on $[0,T]$, with $T\leq mr$, and
	$F:V\to \mathbb{R}$ be a Lipschitzian function. Then, for any $0< t \leq T$ the law of $F(x(t))$ is absolutely continuous with respect to the Lebesgue measure If: 
	$$\int_{t-r}^t \left(F'(x(t))R(t-u)\sigma (x(u-r))\right)^2 du > 0\quad a.s.\,.$$
\end{theorem}
\begin{proof}
	Fix $t\in (0,T]$ such that $\int_{t-r}^t F'(x(t))R(t-u)\sigma (x(u-r))^2 du > 0$ a.s By Proposition 7.1.4 in \cite{boul}, it suffices to show that $F(x(t))\in \mathbb{D}^{1,2}$ and $\|DF(x(t))\|_{L^2([0,T])}>0$ a.s.\\
	Since $x(t)\in \mathbb{D}^{1,2}(V)$ and $F: V\longrightarrow {\mathbb R} $ is a Lipschitz function, then $F\in \mathbb{D}^{1,2}$ (see Proposition 1.2.4 page.29 in \cite{nualart}) and $$D_uF(x(t))=F'(x(t))D_ux(t)$$
	On the other hand, we have 
	\begin{eqnarray*}
		D_ux(t)&=&-g'(x(t-r))D_ux(t-r)1_{u<t-r}+\int_{u+r}^t R(t-s)f'(x(s-r))D_ux(s-r)ds\\
		&+& R(t-u)\sigma(x(u-r))
		+\int_{u+r}^t R(t-s)\sigma'(x(s-r))D_ux(s-r)dB^H(s)
	\end{eqnarray*}
	Hence, for  $u\in (t-r,t)$, we have
	\begin{equation}\label{d}
		D_ux(t)= R(t-u)\sigma(x(u-r))
	\end{equation}
	Note that $$\|DF(x(t))\|_{L^2([0,T])}>0 \; \mbox{a.s}\;
	\Leftrightarrow \int_0^T \left(F'(x(t))R(t-u)\sigma(x(u-r))\right)^2 du>0 \; \mbox{a.s}\;$$
	and a sufficient condition for this is that
	$$\int_{t-r}^t \left(F'(x(t))R(t-u)\sigma(x(u-r))\right)^2 du>0 \quad \mbox{a.s}\;$$
	which imply that the law of $F(x(t))$ has a density with respect to the Lebesgue measure. This completes the proof.
\end{proof}
\begin{example}
	In the case of $F(x)=\vert\vert x\vert\vert $, we get that the law of $\vert\vert x(t)\vert\vert_{V} $ is absolutely continuous with respect to the Lebesgue measure if
	$$\int_{t-r}^t <x(t)\,,\, R(t-u)\sigma(x(u-r))>_{V}^2 du>0 \quad \mbox{a.s}\,.$$
\end{example}

\section{Application}
We consider the following stochastic partial neutral functional integro-differential equation with finite delay $r$:
\begin{equation}\label{Eq2}
	\left\{
	\begin{array}{rl}
		\frac{\partial}{\partial t}\left[ x(t,y)+\hat{g}(x(t-r,y))\right] &=\frac{\partial^2}{\partial y^2}\left[ x(t,y)+\hat{g}(x(t-r,y))\right] \\
		&+\int_0^tb(t-s)\frac{\partial^2}{\partial y^2}\left[ x(t,y)+\hat{g}(x(s-r,y))\right] ds\\
		&+\hat{f}(x(t-r,y))+\hat{\sigma}(x(t-r,y))\frac{dB^H}{dt}(t)\\
		
		x(t,0)+\hat{g}(x(t-r,0))=0,\ t\ge 0\\
		x(t,\pi)+\hat{g}(x(t-r,\pi))=0,\ t\ge 0\\
		x(t,y)=\varphi(t,y),\ -r\le t\le 0,\, 0\le y\le \pi
	\end{array}
	\right.
\end{equation}
where $\hat{g}, \hat{f}, \hat{\sigma}: \mathbb{R}\to \mathbb{R}$, and $b:\mathbb{R}\to \mathbb{R}$ are continuous functions, such that $\hat{g}(0)=\hat{f}(0)=0$ . Let $X=V=L^2([0,\pi])$ and define the operator $A:D(A)\to V$ by $Az=z"$ with domain
$$D(A)=\{z\in V; z,z'\  are\  absoluty\  continuous, z"\in V, z(0)=z(\pi)=0\}$$
It is well known that $A$ generates a strongly continuous semigroup $\{T(t),\, t\ge 0\}$ on $V$ which is given by:
$$T(t)\varphi=\sum_{n=1}^{\infty} e^{-n^2 t}<\varphi,e_n>\, e_n$$
where $e_n(x)=\sqrt{\frac{2}{\pi}}\sin(nx)$ is the orthogonal set of eigenvectors of $-A$.\\
Define the operators $g,\, f$, and $\sigma:V\to V$ by:
\begin{eqnarray*}
	g(\psi)(y)&=& \hat{g}(\psi(y)),\ for\  \psi \in V\  and\  y\in [0,\pi]\\
	f(\psi)(y)&=& \hat{f}(\psi(y)),\ for\  \psi \in V\  and\  y\in [0,\pi]\\
	\sigma(\psi)(y)&=& \hat{\sigma}(\psi(y)),\ for\  \psi \in V\  and\  y\in [0,\pi]\\
\end{eqnarray*}
If we put 
$$x(t)(y)=x(t,y),\quad for \ t\ge -r\quad and\quad y\in [0,\pi]$$
Let $B:D(A)\subset V\to V$ be defined by 
$$B(t)y=b(t)Ay,\quad for\; t \ge 0\quad and\; y\in D(A)$$
then the equation $(\ref{Eq2})$ becomes 
\begin{equation}
	\left\{
	\begin{array}{rl}
		d\left[ x(t)+g(x(t-r))\right]& = A\left[ x(t)+g(x(t-r))\right] dt\\
		&+\left[ \int_0^tB(t-s)\left[ x(s)+g(x(s-r))\right] ds+f(x(t-r))\right] dt\\
		&+\sigma(x(t-r))dB^H(t)\\
		x(t)&=\varphi,\ t\in [-r,0]
	\end{array}
	\right.
\end{equation}
Moreover, if $b$ is a bounded and $C^1$ function such that $b'$ is bounded and uniformly continuous, then $(\mathcal{H}.1)$ and $(\mathcal{H}.2)$ are satisfied and hence by theorem \ref{res}, equation $(\ref{Eq2})$ has a resolvent operator $(R(t),\ t\ge 0)$ on $V$.\\
Further, if we impose suitable regularity conditions on $\hat{g}$, $\hat{f}$ and $\hat{\sigma}$ such that $g$, $f$ and $\sigma$ verify assumptions of theorem \ref{thm1}. then we conclude the existence and uniqueness of the mild solution of the equation $(\ref{Eq2})$.

\end{document}